\documentclass[a4paper,freqn,11pt]{amsart}
\usepackage{amssymb, comment}
\usepackage{mathrsfs}
\usepackage{graphicx}
\usepackage[colorlinks=true]{hyperref}
\numberwithin{equation}{section}

\usepackage[a4paper,bindingoffset=0.5cm,left=2cm,right=2cm,top=2.5cm,bottom=2cm,footskip=.8cm]{geometry}

\def \E {\mathbb E}

\def \P {\mathbb{P}}

\newtheorem{theorem}{Theorem}
\newtheorem{lemma}{Lemma}

\theoremstyle{remark}
\newtheorem{remark}{Remark}
\newtheorem{assumption}{Assumption}

\allowdisplaybreaks

\renewenvironment{proof}[1][\proofname]{\par \normalfont \trivlist
 \item[\hskip\labelsep\itshape #1]\ignorespaces
}{
 \hspace*{\fill}$\Box$ \endtrivlist
}
\renewcommand{\proofname}{\noindent {\bf Proof}}

\begin{document}

\title[Generalized gap acceptance models for unsignalized intersections]{Generalized gap acceptance models\\ for unsignalized intersections}
\author[Abhishek]{Abhishek$^1$}\thanks{$^1$ Korteweg-de Vries Institute for Mathematics, University of Amsterdam, Amsterdam, The Netherlands
({\tt \{Abhishek,m.r.h.mandjes\}@uva.nl})}

\author[Marko Boon]{Marko Boon$^2$}\thanks{
$^2$Department of Mathematics and Computer Science, Eindhoven
University of Technology, P.O. Box 513, 5600 MB Eindhoven, The
Netherlands ({\tt m.a.a.boon@tue.nl})}


\author[Michel Mandjes]{Michel Mandjes$^1$}\thanks{
}


\date{\today}
\maketitle
\begin{abstract}
This paper contributes to the modeling and analysis of unsignalized intersections. In classical gap acceptance models vehicles on the minor road accept any gap greater than the \emph{critical} gap, and reject gaps below this threshold, where the gap is  the time between two subsequent vehicles on the major road.
The main contribution of this paper is to develop a series of generalizations of  existing models, thus increasing the model's practical applicability significantly.
First, we incorporate {driver impatience behavior} while allowing for a realistic merging behavior; we do so by distinguishing between the critical gap and the merging time, thus allowing {\it multiple} vehicles to use a sufficiently large gap. Incorporating this feature is particularly challenging in models with driver impatience.
Secondly, we allow for multiple classes of gap acceptance behavior, enabling us to distinguish between different driver types and/or different vehicle types.
Thirdly, we use the novel M$^X$/SM2/1 queueing model, which has batch arrivals, dependent service times, and
a different service-time distribution for vehicles arriving in an empty queue on the minor road (where `service time' refers to the time required to find a sufficiently large gap). This setup facilitates the analysis of the service-time distribution of an arbitrary  vehicle on the minor road and of the queue length on the minor road. In particular, we can compute the {\it mean} service time, thus enabling the evaluation of the
capacity for the minor road vehicles.

\vspace{1mm}

\end{abstract}
\section{Introduction}\label{sect:intro}

In this era of rapidly emerging new technologies for urban traffic control, the vast majority of urban traffic intersections are still priority-based unsignalized intersections, in which the major roads have priority over the minor roads. Due to the increasing level of traffic congestion, the need for state-of-the-art quantitative analysis methods is greater than ever. Performance analysis of unsignalized traffic intersections is traditionally based on so-called \emph{gap acceptance models}, with roots that can be traced back to classical queueing models. The basis of these models is the assumption that the crossing decision of a driver on the minor road is based on the gap between two successive vehicles on the major road \cite{drew3}. This model can be applied in other contexts as well, e.g.\ when analyzing  freeways \cite{drew2,drew1,munjal} and pedestrian crossings \cite{tanner51,wei}.\\

In this paper we contribute to the modeling and analysis of unsignalized intersections. In classical gap acceptance models vehicles on the minor road accept any gap greater than the \emph{critical} gap (or critical headway), and reject gaps below this threshold, where the gap is defined as the time between two subsequent vehicles on the major road. However, in reality,  drivers typically do not need the full critical gap; the remaining part can be used  by the next vehicle on the minor road. In the literature, the part of the critical gap that is really used by the vehicle is referred to as the \emph{merging} time.

Although gap acceptance models have improved greatly since their introduction more than fifty years ago, there are still some fundamental limitations to their practical applicability. As pointed out recently by Liu {\it et al.} \cite{liu2014}, two of its main limitations are the inability to incorporate (1)~different driver behaviors and (2)~heterogeneous traffic; see also \cite{prasetijo2007,prasetijo2012}. The main reason why these features could not be included is that the currently used analysis techniques are based on a single-server queueing model with exceptional first service (see also \cite{welch,yeo,yeoweesakul}). In the gap acceptance literature, this model is commonly referred to as the M/G2/1 queue, a term seemingly introduced by Daganzo \cite{daganzo1977}.
Importantly, in this queue, the `service times' (corresponding to the time required to search for a sufficiently large gap and crossing the intersection or, depending on the application, merging with the high-priority traffic flow) are assumed independent. Incorporating more realistic features, such as driver impatience or different types of driver behavior, creates dependencies that make the model significantly more difficult to analyze.

An important development, that helps us overcome this obstacle, is the recently developed framework for the analysis of queueing models that allows a highly general correlation structure between successive service times (the M$^X$/SM2/1 queue  \cite{AbhishekHeavyTraffic}). Although set up with general applications in mind, this framework turns out to be particularly useful in road traffic models, where the dependencies can be used to model clustering of vehicles on the major road, to differentiate between multiple types of driver behavior, or to account for heterogeneous traffic. The present paper focuses on the latter two features.\\

In the existing literature, various quantitative methods have been used to study unsignalized intersections. The two most common procedures are empirical regression techniques and gap-acceptance models \cite{brilon1997}, but other methods have been employed as well (for example the additive conflict flows technique by Brilon and Wu \cite{brilonwu}). A topic of particular interest concerns the capacity of the minor road \cite{brilon}, which is defined as the maximum possible number of vehicles per time unit  that can pass through an intersection from the minor road. Other relevant performance measures are the queue length and the delay on the minor road.

Due to the driver characteristics such as age, gender, years of driving experience, as well as the vehicle's features, in particular the size and acceleration speed of the vehicle,  there are different types of driver behavior. First, we distinguish between fast vehicles and slow vehicles. The differences in speed/acceleration can be caused by human factors (risk avoidance) or by different vehicle types (e.g.\ cars, trucks). Second, we assume that drivers grow more and more impatient while waiting for a sufficiently large gap on the major road, which was actually confirmed in an empirical data study by Abou-Henaidy {\it et al.} \cite{abouhenaidy94}. Only very few papers on gap acceptance models incorporate driver impatience in the model (\cite{abhishekcomsnets2016,drew2,drew1,weissmaradudin}), because it is known to complicate the analysis drastically \cite{heidemann97}, in particular when trying to maintain a realistic merging behavior.

In earlier papers, such as \cite{abhishekcomsnets2016}, three variations of gap acceptance models can be distinguished. The first  is the basic model in which all low-priority drivers are assumed to use the same fixed critical gap. In the second model,  critical gaps are random, with drivers sampling a new critical gap at each new attempt. This is typically referred to as \emph{inconsistent} driver behavior. The third model is also known as the \emph{consistent} model, in which a random critical gap is sampled by each driver for his first attempt only, and the driver then uses that same value at his subsequent attempts. The present paper generalizes all three gap acceptance models into a more realistic model that covers various realistic driving behavior features.

One of the first studies using queueing models for unsignalized intersections was by Tanner \cite{tanner62}, who assumed  constant critical gap and move-up time. Tanner first determines the mean delay of the low priority vehicles, and characterizes the capacity of the minor road as the arrival rate of the low-priority vehicles, at which the mean delay grows beyond any bound. Tanner's model has been generalized in various ways  \cite{catchpoleplank,cheng,hawkes,heidemann94,wegmann1991,weissmaradudin}  by allowing random critical gaps and move-up times, also analyzing performance measures such as the queue length and the waiting time on the minor road. Heidemann and Wegmann \cite{heidemann97} give an excellent overview of the earlier existing literature. Moreover, they add a stochastic dependence between the critical gap and the merging (move-up) time, and study the minor road as an M/G2/1 queue, to determine the queue length, the delay and the capacity on the minor road. A further generalization, dividing the time scale of the major stream into four regimes (viz.\ free space, single vehicle, bunching, and queueing) was investigated by Wu \cite{wu2001}.\\


Although the literature on gap acceptance models is relatively mature and the existing gap acceptance models have proven their value, they are sometimes criticized for \emph{looking like} rigorous mathematics, but in reality being based on pragmatic simplifications. As a consequence, the produced results might be of  a  correct magnitude, but are, however, only of approximative nature~\cite{brilonwu}. The main goal of this paper is to increase the significance of gap acceptance models by taking away some of these concerns by including essential new features that are typically encountered in practice.
In more detail, the main novelties of this paper are the following.
\begin{itemize}
  \item First, we incorporate {driver impatience behavior} while allowing for a realistic merging behavior. Put more precisely, we distinguish between the critical gap and the merging time, thus allowing {\it multiple} vehicles to use a sufficiently large gap. Incorporating this feature is particularly challenging in models with driver impatience.
  \item Secondly, we allow for multiple classes of gap acceptance behavior, enabling us to distinguish between different driver types and/or different vehicle types.
  \item Thirdly, we use a queueing model in which vehicles arriving in an empty queue on the minor road have different service-time distributions than the queueing vehicles (where `service time' is meant in the sense introduced above). This setup facilitates the analysis of the queue length on the minor road as well as the service-time distribution of an arbitrary low-priority vehicle. The capacity for the minor road vehicles is then derived from the expectation of the service time for queuers.
\end{itemize}

The remainder of this paper is organized as follows. In Section \ref{model description} we introduce our model. Then the full queue-length analysis is presented in Section \ref{Queing length}, whereas the capacity of the minor road is determined in Section \ref{capacity on minor road}. In Section \ref{numerics} numerical examples demonstrate the impact of driver behavior on the capacity of 
the minor road.

\section{Model description}\label{model description}

In this section we provide a detailed mathematical description of our new model. We study an unsignalized, priority controlled intersection where drivers on the major road have priority over the drivers on the minor road. For notational convenience, we will focus on the situation with one traffic stream on each road, but several extensions also fall in our framework (see, for example, Figure  \ref{fig:intersection}).
Vehicles on the major road arrive according to a Poisson process with intensity $q$. On the minor road, we have a \emph{batch} Poisson arrival process with $\lambda$ denoting the arrival intensity of the batches (platoons) and $B$ denoting the (random) platoon size. We assume that traffic on the major road is not hindered by traffic on the minor road, which is a reasonable assumption in most countries (but not all, see \cite{liu2014}).

\begin{figure}[!ht]
\begin{center}
\includegraphics[width=0.7\linewidth]{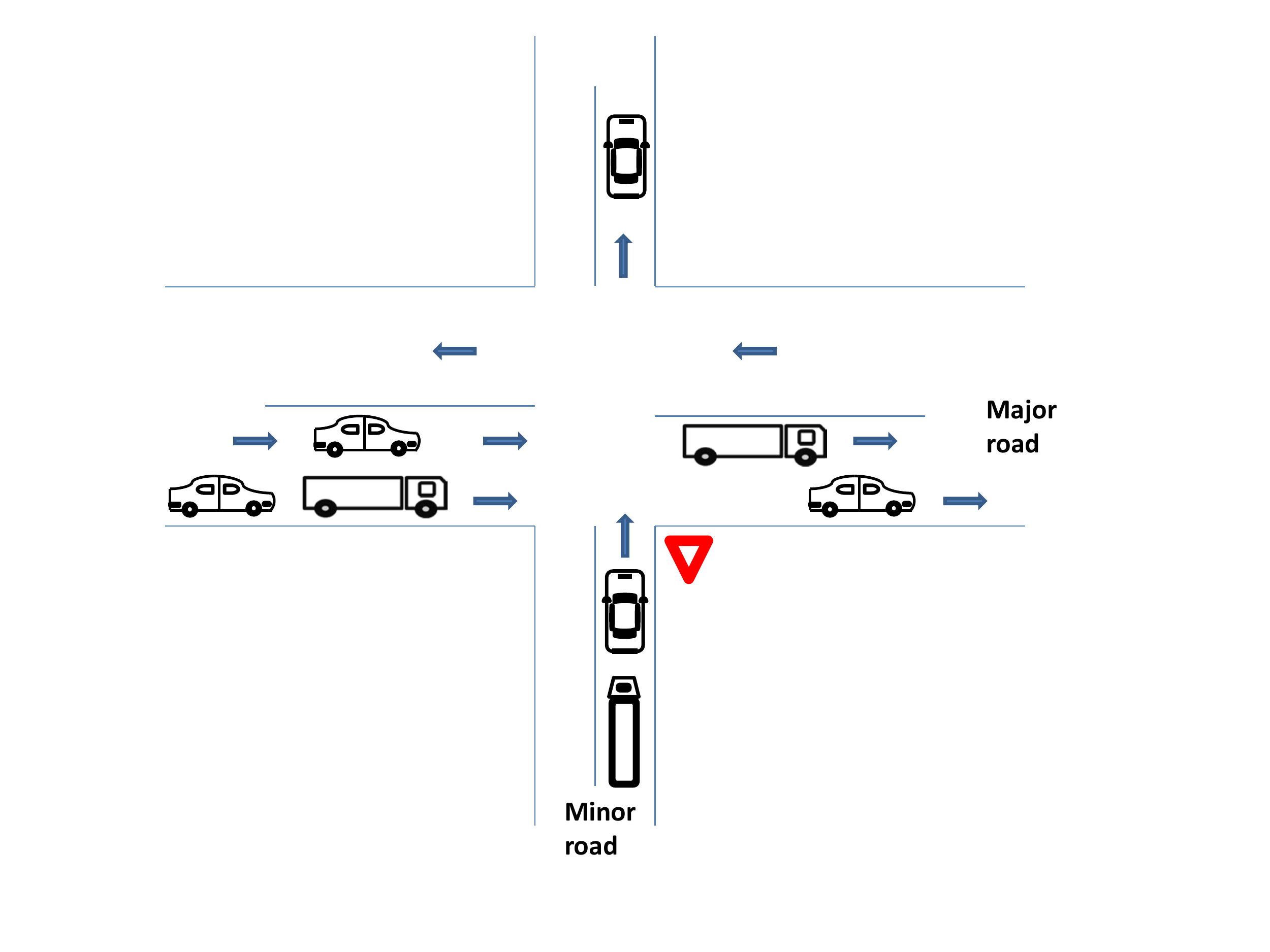}
\end{center}
\caption{The situation analyzed in this paper.
}
\label{fig:intersection}
\end{figure}

The drivers on the minor road cross the intersection as soon as they come across a sufficiently large gap between two subsequent vehicles on the major road. Any gap that is too small for a driver on the minor road is considered a \emph{failed attempt} to cross the intersection and adds to the impatience of the driver.\\

In the existing literature, \emph{three} model variants have been introduced: (1) the standard model with constant critical gaps, (2) inconsistent gap acceptance behavior where a new random critical gap is sampled at each attempt, and (3) consistent behavior where drivers sample a random critical gap which they use for all attempts (cf.\ \cite{abhishekcomsnets2016,heidemann97}).
In this paper we apply a framework that allows us to create a \emph{single} {model} that generalizes all three aforementioned model variations, with a combination of consistent \emph{and} inconsistent behavior, allowing for heterogeneous traffic and driver impatience. This is a major enhancement of the existing models, as such a general model has not been successfully analyzed so far. Unfortunately, this requires a significant adaptation of the underlying M/G2/1 queueing model that has been the basis of all gap acceptance models so far. \\

We will now describe the full model dynamics in greater detail. We distinguish between $R\in\{1,2,\ldots\}$ driver/vehicle profiles. The profile of any arriving vehicle on the minor road is modeled as a random variable that equals $r$ with probability $p_r$ (with $\sum_{r=1}^R p_r=1$). The profile represents the \emph{consistent part} of the gap selecting behavior.
Each driver profile has its own critical gap acceptance behavior, allowing us to distinguish between drivers that are willing to accept small gaps and (more cautious) drivers that require longer gaps, but also to allow for heterogeneous traffic with, for example, trucks requiring larger gaps than cars or motor cycles.

We introduce $T_{(i,r)}$ to denote the critical gap of a driver with profile $r$ during his $i$-th attempt. The \emph{impatient behavior} is incorporated in our framework by letting the critical gap $T_{(i,r)}$ depend on the attempt number $i=1,2,3,\dots$. This generalizes classical gap acceptance models where the critical gap is the same for all driver types and throughout all attempts.

Finally, we introduce the \emph{inconsistent} component of our model by allowing the critical gap of a driver of type $r$ in attempt $i$ ($T_{(i,r)}$, that is)  to be a discrete random variable which can take any of the following values:
\[
T_{(i,r)}=\begin{cases}
u_{(i,1,r)} &\quad \text{ with probability }p_{(i,1,r)},\\
u_{(i,2,r)} &\quad \text{ with probability }p_{(i,2,r)},\\
\vdots&\\
u_{(i,M,r)} &\quad \text{ with probability }p_{(i,M,r)},
\end{cases}
\]
with $\sum_{k=1}^M p_{(i,k,r)}=1$ for all $(i,r)$. Inconsistent behavior is used in the existing literature when the driver samples a new random critical gap after each failed attempt. It is generally being criticized for not being very realistic, because it is unlikely that a driver's gap acceptance behavior fluctuates significantly throughout multiple attempts. In our model, however, it is an excellent way to include some randomness in the gap selection process (which is due to the fact that not every driver from the same profile will have \emph{exactly} the same critical gap at his $i$-th attempt), because we can limit the variability by cleverly choosing the values $u_{(i,k,r)}$.

Realistic values for the critical gaps $u_{(i,k,r)}$ can be empirically obtained by measuring the lengths of all rejected and accepted gaps for each vehicle type. One needs to distinguish between systematic fluctuations due to different driver/vehicle types and (smaller) random fluctuations within the same driver profile. Since the index $i$ represents impatience, it makes sense that $u_{(i,k,r)}$ is decreasing in $i$. Moreover, since $r$ denotes the driver profile, $u_{(i,k,r)}$ is relatively large for driver profiles $r$ that represent slow vehicles/drivers and vice versa. The variability in $u_{(i,1,r)}, \dots, u_{(i,M,r)}$ can be limited to avoid unrealistic fluctuations. For example: if driver profile $r=1$ corresponds to trucks and $r=2$ corresponds to motor cycles, then it may be quite realistic that \[\min_{1\leq k\leq M} u_{(i,k,1)} >  \max_{1\leq k\leq M} u_{(i,k,2)}.\] Paraphrasing, the slowest motor cycle will always be faster than the fastest truck.

To make the model even more realistic, we assume that the actual vehicle \emph{merging time}, denoted by $\Delta_r$ for profile $r$, is less than the critical gap $T_{(i,r)}$. As a consequence, the remaining part of the critical headway can be used by following drivers.

At first sight, our model has a few seeming limitations. In the first place, we assume the critical gap has a discrete distribution with $M$ possible values. In the second place, as will be discussed in great detail in the next section, the analysis technique requires that there is a maximum number of possible attempts: $i\in\{1,2,\dots,N\}$. It is important to note, however, that both issues do not have big practical consequences, since one can choose $M$ and $N$ quite large. This does come at the expense of increased computation times, obviously.

\section{ Queue length analysis}\label{Queing length}
The main objective of this section is to determine the stationary queue length on the minor road at departure epochs of low-priority vehicles. In the next section we will use these results to determine the service-time distribution of the low-priority vehicles, which enables us to evaluate the capacity of the minor road.

\subsection{Preliminaries}

We start by describing the queueing process on the minor road.
All arriving vehicles experience three stages before leaving the system: queueing, scanning and merging. The queueing phase is defined as the time between joining the queue and reaching the front of the queue (for queuers this is the moment that the preceding driver accepts a gap and departs). The scanning phase commences when the driver starts scanning for gaps and ends when a sufficiently large gap is found. The scanning phase is followed by the merging phase, which ends at the moment that the vehicle merges into the major stream. Note that phases 1 and 2 may have zero length. The length of phase 1 is called the \emph{delay}. The length of phase 3 for a vehicle of profile $r$ is denoted by the profile-specific constant $\Delta_r$. We define the \emph{service time} as the total time spent in phases 2 and 3 (scanning and merging).

\begin{remark}
In some papers, an additional move-up phase is distinguished but, depending on the application (freeway or unsignalized intersection), this phase can be incorporated in one of the other phases. See Heidemann and Wegmann \cite[Section 3]{heidemann97} for more details.
\end{remark}

The queue-length analysis is based on observing the system at \emph{departure epochs}. In contrast to the classical M/G2/1 queueing models, in our current model we now have \emph{dependent} service times. This dependence is caused by the fact that the length of the lag (i.e., the remaining gap) left by the previous driver, due to the impatience, now depends on the attempt number \emph{and} the profile of the previous driver. Using the notation introduced in the previous section: we need to know the critical gap $T_{(i, r)}$ of the previous driver of profile $r$ that led to a successful merge/crossing of the intersection. A queueing model with this type of dependencies is called a queueing model with \emph{semi-Markovian} service times, sometimes referred to as the M/SM/1 queue. Although several papers have studied such a queueing model (cf.\ \cite{QUESTA2017,AK,cin_s,desmit86,gaver,neuts66,neuts77b,neuts77a,pur}), we still need to make several adjustments in order to make it applicable to our situation. The analysis below is based on the framework that was recently introduced in \cite{QUESTA2017}.

Denote by $X_n$ the queue length at the departure epoch of the $n$-th vehicle, i.e. the number of vehicles \emph{behind} the $n$-th driver at the moment that he merges into the major stream. We use $G^{(n)}$ to denote the service time of the $n$-th vehicle and $A_n$ to denote the number of arrivals (on the minor road) during this service time. Since vehicles arrive in batches (platoons), we denote by $B_n$ the size of the batch in which the $n$-th vehicle arrived. Let $X(\cdot)$, $A(\cdot)$, and $B(\cdot)$ denote, respectively, the probability generating functions (PGFs) of the limiting distributions of these random variables,
\[
X(z)=\lim_{n\to\infty}\E[z^{X_n}], \quad A(z)=\lim_{n\to\infty}\E[z^{A_n}], \quad B(z)=\lim_{n\to\infty}\E[z^{B_n}], \qquad |z| \leq 1,
\]
whereas $\tilde{G}(\cdot)$ is the limiting Laplace-Stieltjes transform (LST) of $G^{(n)}$:
\[\tilde{G}(s) =\lim_{n\to\infty}\E[e^{-s G^{(n)}}], \qquad s\geq 0.\]
Note that the arrivals 
constitute a batch Poisson process. Therefore, $A(z)$ can be expressed in terms of $B(z)$ and $\tilde{G}(s)$,
\begin{equation}\label{eqn:pgfAz}
A(z)=\tilde{G}(\lambda(1-B(z))),
\end{equation}
but finding an expression for $\tilde{G}(s)$ is quite tedious. For this reason, we split the analysis in two parts. This section discusses how to find $X(z)$ in terms of $A(z)$, while Section \ref{capacity on minor road} shows how to find $G(z)$ for the gap acceptance model considered in this paper.

The starting-point of our analysis is the following recurrence relation:
\begin{align}
X_n = \left\{
\begin{array}{l l}
X_{n-1}-1+A_n & \quad \text{if $X_{n-1} \geq 1$ }\\
A_n +B_n-1& \quad \text{if $X_{n-1} =0$}
\end{array} \right., ~~~ n=1,2,3,\dots.
\label{recurA}
\end{align}
As argued before, $\{X_n\}_{n\geq1}$ is not a Markov chain. In order to obtain a Markovian model, we need to keep track of all the characteristics of the $(n-1)$-th driver. Let $J_n= (i,k, r)$ contain these characteristics of the $(n-1)$-th driver, where $i$ is the `succeeded attempt number', $k$ indicates which random gap the driver sampled, and $r$ denotes the profile of the driver. 

In \cite{AbhishekHeavyTraffic} it is shown that the ergodicity condition of this system is given by
\begin{align}
\lim_{n\to\infty} \E[A_n|X_{n-1}\geq 1]<1. \label{stability}
\end{align}
Throughout this paper we assume that this condition holds.

\subsection{A queueing model with semi-Markovian service times and exceptional first service}

To find the probability generating function (PGF) of the queue-length distribution at departure epochs, we use the framework introduced in \cite{AbhishekHeavyTraffic}. There an analysis is presented of the M$^X$/SM2/1 queue, which is a very general type of single-server queueing model with batch arrivals and semi-Markovian service times with exceptional service when the queue is empty, extending earlier results (cf.\ also \cite{QUESTA2017,cin_s,gaver,neuts66}) to make it applicable to situations as the one discussed in the present paper. To improve the readability of this paper, we briefly summarize the most important results from \cite{AbhishekHeavyTraffic} that are also valid for our model.

In \cite{AbhishekHeavyTraffic} the type of the $n$-th customer is denoted by $J_{n}\in \{1,2, \dots, N\}$. In order to apply these results correctly to our model, we need to make one small adjustment. In order to determine the service time of the $n$-th customer, we need to know exactly how much of the critical gap of the $(n-1)$-th customer remains for this $n$-th customer.
As it turns out in the next section, this requires the following knowledge about the $(n-1)$-th customer:
\begin{enumerate}
\item[(A)] we need to know the sampled critical gap $u_{(i,k,r)}$ of the $(n-1)$-th customer;
\item[(B)] we need to know whether the $(n-1)$-th customer emptied the queue at the minor road and the $n$-th customer was the first in a new batch arriving some time after the departure of his predecessor.
\end{enumerate}

To start with the latter: a consequence of (B) is that we need a queueing model with so-called \emph{exceptional first service}.
As a consequence of (A), we use the triplet $(i, k, r)$, extensively discussed in the previous subsection, to denote the `customer type', which should be interpreted as a vehicle of profile $r$ that accepted the $i$-th gap, where it sampled $u_{(i,k,r)}$ as its critical headway. It means that if the $n$-th customer is of type $(i,k,r)$, then his \emph{predecessor} (which is the $(n-1)$-th customer) succeeded in his $i$-th attempt, while being of profile $r$ and having critical gap $u_{(i,k,r)}$.


A translation from our model to the model in \cite{AbhishekHeavyTraffic} can easily be made by mapping our $\bar N:=NMR$ customer types onto their $N$ customer types (for instance, if $j=(i-1)MR+(k-1)R+r$, then $j\in\{1,2,\dots,\bar N\}$ can serve as the customer type in \cite{AbhishekHeavyTraffic}).

Mimicking the steps in \cite{AbhishekHeavyTraffic}, it is immediate that
\begin{align}
\label{twentysixAA_BB}
\E[z^{X_{n}}&1_{\{J_{n+1}=(j,l,r_1)\}}]\\
=&\,\frac{1}{z}\sum_{r_0=1}^R\sum_{k=1}^{M}\sum_{i=1}^{N} \E[z^{X_{n-1}}1_{\{J_{n}=(i,k,r_0)\}}]\E[z^{A_{n}}1_{\{J_{n+1}=(j,l,r_1)\}}|J_{n}=(i,k,r_0),X_{n-1} \geq 1]\,+\nonumber\\
\ \ \ & \,\frac{1}{z}\sum_{r_0=1}^R\sum_{k=1}^{M}\sum_{i=1}^{N}\Big(B(z)\E[z^{A_{n}}1_{\{J_{n+1}=(j,l,r_1)\}}|J_{n}=(i,k,r_0),X_{n-1}=0]\,-\nonumber\\
&\,\E[z^{A_{n}}1_{\{J_{n+1}=(j,l,r_1)\}}|J_{n}=(i,k,r_0),X_{n-1}\geq 1]\Big)\P(X_{n-1}=0,J_{n}=(i,k,r_0)),\nonumber
\end{align}
for $n\in {\mathbb N}$, $ j\in\{1,2,\dots,N\}$, $l\in\{1,2,\dots,M\}$, and $r_1\in\{1,2,\dots,R\}$. 

To further evaluate these expressions, we need to introduce some additional notation.
Define, for $i,j\in\{1,2,\dots,N\}$ $k,l\in\{1,2,\dots,M\}$, $r_0,r_1\in\{1,2,\dots,R\}$,  and $|z| \leq 1$,

\begin{align*}
A^{(j,l,r_1)}_{(i,k,r_0)}(z)&=\lim_{n\to\infty}\E[z^{A_{n}}1_{\{J_{n+1}=(j,l,r_1)\}}|J_n=(i,k,r_0),X_{n-1}\geq 1],\\
A^{*(j,l,r_1)}_{\:\:(i,k,r_0)}(z)&=\lim_{n\to\infty}\E[z^{A_{n}}1_{\{J_{n+1}=(j,l,r_1)\}}|J_n=(i,k,r_0),X_{n-1}=0],
\end{align*}
where $A^*$ corresponds to the number of arrivals during the (exceptional) service of a vehicle arriving when there is no queue in front of him. In addition,
\begin{align}
f_{(j,l,r_1)}(z)&= \lim_{n \rightarrow \infty}\E[z^{X_n}1_{\{J_{n+1}=(j,l,r_1)\}}],
\label{fizA_BB}
\end{align}
such that
\begin{align}
f_{(j,l,r_1)}(0)&= \lim_{n \rightarrow \infty} \P(X_n=0,J_{n+1}=(j,l,r_1)).
\label{fi0A_BB}
\end{align}
The above definitions entail that
\begin{equation}
X(z)=\sum_{r_1=1}^R\sum_{l=1}^{M}\sum_{j=1}^{N}f_{(j,l,r_1)}(z). \label{F(z)_BB}
\end{equation}
In steady state, Equation \eqref{twentysixAA_BB} thus leads to the following system of $\bar N$ linear equations: for any $j,$ $l$, and $r_1$,
\begin{align}\label{twentysevenA_BB}
 &\left(z-A^{(j,l,r_1)}_{(j,l,r_1)}(z)\right)f_{(j,l,r_1)}(z)-\sum_{\substack{r_0=1,\\r_0\neq r_1}}^RA^{(j,l,r_1)}_{(j,l,r_0)}(z)f_{(j,l,r_0)}(z)\\
 &-\sum_{r_0=1}^R\sum_{\substack{k=1,\\k\neq l}}^{M}A^{(j,l,r_1)}_{(j,k,r_0)}(z)f_{(j,k,r_0)}(z)
-\sum_{r_0=1}^R\sum_{k=1}^{M}\sum_{\substack{i=1,\\i\neq j}}^{N}A^{(j,l,r_1)}_{(i,k,r_0)}(z)f_{(i,k,r_0)}(z)\nonumber\\
&=\sum_{r_0=1}^R\sum_{k=1}^{M}\sum_{i=1}^{N}(B(z)A^{*(j,l,r_1)}_{\:\:(i,k,r_0)}(z)-A^{(j,l,r_1)}_{(i,k,r_0)}(z))
f_{(i,k,r_0)}(0).\nonumber
\end{align}
From the above we conclude that when we know $A^{(j,l,r_1)}_{(i,k,r_0)}(z)$ and $A^{*(j,l,r_1)}_{\:\:(i,k,r_0)}(z)$,
the PGFs of our interest can be evaluated. In the next section we show how to find the Laplace-Stieltjes transforms of the conditional service-time distributions, which lead to $A^{(j,l,r_1)}_{(i,k,r_0)}(z)$ and $A^{*(j,l,r_1)}_{\:\:(i,k,r_0)}(z)$. 

We refer to \cite[Section 3]{AbhishekHeavyTraffic} for more details on how to write the linear system of equations (\ref{twentysevenA_BB}) in a convenient matrix form and solve it numerically. In \cite{AbhishekHeavyTraffic}, and in more detail in \cite{QUESTA2017}, it is also discussed how the PGF of the queue-length  at \emph{arbitrary epochs} (denoted by $X^{\rm arb}$) can be found; in particular the clean relation between $X^{\rm arb}$ and $X$:
\begin{equation}
 \E[z^{X^{\rm arb}}]=\E[z^{X}] \frac{\E[B](1-z)}{1-B(z)}
 \label{F(z):arbitrary}
\end{equation}
is useful in this context.
The interpretation is that the number of customers at an \emph{arbitrary epoch} is equal to the number of customers left behind by an arbitrary \emph{departing} customer, excluding those that arrived in the same batch as this departing customer. We refer to \cite{QUESTA2017} for more details and a proof of \eqref{F(z):arbitrary}.

\section{Capacity}\label{capacity on minor road}

In this section we derive the Laplace-Stieltjes transform of the service-time distribution of our gap acceptance model. This service-time analysis serves two purposes. Firstly, it is used to complete the queue-length analysis of the previous section. Secondly, it is an important result by itself, because it is directly linked to the capacity of the intersection, which is defined as the maximum possible number of vehicles per time unit that can pass from the minor road. An alternative but equivalent definition of capacity, also employed by Heidemann and Wegmann \cite{heidemann97}, is the maximum arrival rate for which the corresponding queue remains stable. Combining \eqref{eqn:pgfAz} and \eqref{stability},
and denoting
\[g:=\lim_{n\to\infty}\E[G^{(n)}\,|\,X_{n-1}\geq 1],\]
we can rewrite the stability condition as
\begin{equation}\label{stabilityAlt}
\lambda\E[B]\,g<1,
\end{equation}
where $\lambda$ is the arrival rate of batches, implying that the arrival rate of individual vehicles is $\lambda\E[B]$. As a consequence, the capacity of the minor road is given by
\begin{align}
C=\frac{1}{g}.\label{capacity}
\end{align}

Every driver (of profile $r$) samples a random $T_{(i,r)}$ at his $i$-th attempt, irrespective of his previous attempts. However, he uses only a constant $\Delta_r$ of the gap to cross the main road. As a consequence, when accepting the gap, the remaining part $(T_{(i,r)}-\Delta_r)$ can be used by subsequent drivers. In order to keep the analysis tractable, we need the following assumption.

\begin{assumption}[Limited gap reusability assumption]\label{assumptiongapuse}
We assume that \emph{at most one} driver can reuse the remaining part of a critical gap accepted by his predecessor. The means that $(T_{(i,r)}-\Delta_r)$ should \emph{not} be large enough for more than one succeeding vehicle to use it for his own critical headway.
\end{assumption}

To avoid confusion, we stress that this assumption does \emph{not} mean that a large gap between successive vehicles on the major road cannot be used by more than one vehicle from the minor road. It \emph{does} mean, however, that the difference between \emph{critical} gaps of two successive vehicles cannot be so large that the second vehicle's critical gap is completely contained in its predecessor's critical gap (excluding the merging time).

\begin{remark}
This assumption seems to be realistic in most, but not all, cases, because one can envision for instance situations in which (parts of) a critical gap accepted by, say, a slow truck might be reused by two fast accelerating vehicles following this truck. It is noted, however, that if Assumption \ref{assumptiongapuse} is violated, despite our method no longer being exact, the computed capacity is still very close to the real, simulated capacity;
numerical evidence of this property is provided in Section \ref{numerics}.
It implies that our method can still be used as a very accurate approximation in those (rare) cases where Assumption \ref{assumptiongapuse} is violated.
\end{remark}

We recall that there is a complicated dependence structure between service times of two successive vehicles, due to our assumption that a vehicle can use part of the gap left behind by its predecessor \emph{combined} with the assumption that drivers have different profiles and tend to become impatient. This obstacle can be overcome by analyzing the  service time  of a vehicle of type $(j, l, r_1)$ in the situation that its predecessor was of type $(i, k, r_0)$.

We proceed by introducing some notation. Let $E_y$ be the event that there is a predecessor gap available of length $y$. Denote by $\pi_{(i,k,r_0)}$ the probability that an arbitrary driver with profile $r_0$ succeeds in his $i$-th attempt with critical headway $u_{(i,k,r_0)}$. In addition, let
\begin{align}
\tilde{G}_{(j,l,r_1)}(s,y)&=\E\left[e^{-sG^{(n)}}1_{\{J_{n+1}=(j,l,r_1)\}}\,|\,E_y\right],\label{G_(j,l)(s,y)_B}
\end{align}
for $0\leq y \leq u_{(i,k,r_0)}-\Delta_{r_0}$, $j=1,2,\dots,N$, and $l=1,2,\dots,M$. The LST of an arbitrary service time can be found by conditioning on the type of the current vehicle and its predecessor:
\begin{align}
\tilde{G}(s)=\E[e^{-sG^{(n)}}]=\sum_{r_0=1}^R\sum_{r_1=1}^R\sum_{j=1}^N\sum_{l=1}^M\sum_{i=1}^N\sum_{k=1}^M\tilde{G}^{(j,l,r_1)}_{(i,k,r_0)}(s)\P(J_n=(i,k,r_0)),\label{G_B}
\end{align} where \begin{align}\label{G_(j,l,r_1)_(i,k,r_0)}
                    \tilde{G}^{(j,l,r_1)}_{(i,k,r_0)}(s)=\E\left[e^{-sG^{(n)}}1_{\{J_{n+1}=(j,l,r_1)\}}|J_n=(i,k,r_0)\right],
                  \end{align}and
                  \begin{align}\label{P(J=(i,k))_B_2}
                  \P(J_n=(i,k,r_0))=\pi_{(i,k,r_0)}p_{(i,k,r_0)}p_{r_0}.
                  \end{align}

We are now ready to state the main result of this paper. Define
\[\bar f_{(i,k,r)}:= \frac{f_{(i,k,r)}(0)}{\pi_{(i,k,r)}p_{(i,k,r)}p_{r}},\:\:\:
\bar u_{(i,k,r)}:= u_{(i,k,r)}-\Delta_{r}.\]
\begin{theorem}\label{thm1}
The partial service-time LST of a customer, jointly with the event that he is of type $(j,l,r_1)$,
given that his predecessor was of type $(i,k,r_0)$, is given by
\begin{align}\label{G_ij_B}
\tilde{G}&^{(j,l,r_1)}_{(i,k,r_0)}(s)\,=\,\tilde{G}_{(j,l,r_1)}(s,\bar u_{(i,k,r_0)})\left(1-\bar f_{(i,k,r_0)}\right)+\\
& \Bigg(\int_{x=0}^{\bar u_{(i,k,r_0)}}\lambda e^{-\lambda x}\tilde{G}_{(j,l,r_1)}(s,\bar u_{(i,k,r_0)} -x){\rm d}x + \tilde{G}_{(j,l,r_1)}(s,0) e^{-\lambda\bar u_{(i,k,r_0)}}\Bigg)\bar f_{(i,k,r_0)},\nonumber
\end{align} where
\begin{itemize}
\item[$\circ$]
$f_{(i,k,r_0)}(0)$ is obtained by solving the system of equations \eqref{twentysevenA_BB}, and
\item[$\circ$]
expressions for $\tilde{G}_{(j,l,r_1)}(s,y)$ and $\pi_{(i,k,r_0)}$ are provided in Lemmas \ref{lemmaGsy} and \ref{lemmapis} below.
\end{itemize}
\end{theorem}
\begin{proof}
This partial LST $\tilde{G}^{(j,l,r_1)}_{(i,k,r_0)}(s)$ can be expressed in terms of $\tilde{G}_{(j,l,r_1)}(s,y)$ by first distinguishing between the case where there was no queue upon arrival of the $n$-th customer ($X_{n-1}=0$), and the case where there was a queue ($X_{n-1}\geq 1$). If the customer arrives in an empty system at time $t$ and the previous arrival took place (merged on the major road) at time $t-x$, then we need to check whether the remaining gap $u_{(i,k,r_0)}-\Delta_{r_0}-x=\bar u_{(i,k,r_0)}-x$ is greater than zero, or not. If the system was not empty at arrival time, the lag is simply $\bar u_{(i,k,r_0)}$ due to Assumption~\ref{assumptiongapuse} (we provide more details on this in Remark \ref{remark: sufficient cond_supermodel}). This, combined with  \eqref{fi0A_BB} and \eqref{P(J=(i,k))_B_2}, yields
\begin{align*}
\tilde{G}^{(j,l,r_1)}_{(i,k,r_0)}(s)=
&\lim_{n\to\infty}\E\left[e^{-sG^{(n)}}1_{\{J_{n+1}=(j,l,r_1)\}}|X_{n-1}\geq 1,J_n=(i,k,r_0)\right]\P(X_{n-1}\geq 1|J_{n}=(i,k,r_0))\, + \nonumber\\
&\lim_{n\to\infty}\E\left[e^{-sG^{(n)}}1_{\{J_{n+1}=(j,l,r_1)\}}|X_{n-1}=0,J_n=(i,k,r_0)\right]\P(X_{n-1}=0|J_{n}=(i,k,r_0))\nonumber\\
=&\,\E\left[e^{-sG^{(n)}}1_{\{J_{n+1}=(j,l,r_1)\}}| E_{(u_{(i,k,r_0)}-\Delta_{r_0})}\right]\left(1-\bar f_{(i,k,r_0)}\right)\,+\nonumber\\
&\Bigg(\int_{x=0}^{\infty}\lambda e^{-\lambda x}\E\left[e^{-sG^{(n)}}1_{\{J_{n+1}=(j,l,r_1)\}}| E_{\max(u_{(i,k,r_0)}-\Delta_{r_0}-x,0)}\right]{\rm d}x \Bigg)\bar f_{(i,k,r_0)},\nonumber\\
\end{align*}
which can easily be rewritten to \eqref{G_ij_B}.
\end{proof}

It thus remains to find expressions for $\tilde{G}_{(j,l,r_1)}(s,y)$ and $\pi_{(i,k,r_0)}$, which we do in the following two lemmas. We denote by $E^{(n)}_{r}$ the event that the $n$-th driver does not succeed in his first attempt, given that he has profile $r$. 

\begin{lemma}\label{lemmaGsy}
The conditional service-time LST of a vehicle, given the event $E_y$, is given by
\begin{align*}
\tilde{G}_{(1,l,r)}(s,y) &= p_{r}p_{(1,l,r)}e^{-q\max\{u_{(1,l,r)}-y,0\}}e^{-s\Delta_{r}},\\
\tilde{G}_{(j,l,r)}(s,y) &= p_{r}p_{(j,l,r)}\left(1-\E\left[e^{-(s+q)\max\{T_{(1,r)}-y,0\}}\right]\right)\nonumber\\
&\:\:\:\times
e^{-(s(y+\Delta_{r})+qu_{(j,l,r)})} \left(\frac{q}{s+q}\right)^{j-1}\prod\limits_{m=2}^{j-1}\left(1-\E\left[e^{-(s+q)T_{(m,r)}}\right]\right) ,
\end{align*}
for $j=2,3,\dots,N$.

\end{lemma}
\begin{proof}
For $j=1$, the expression $\tilde{G}_{(j,l,r)}(s,y)$ simply follows from computing the probability that a customer is of type $r$ \emph{and} merges successfully during its first attempt with critical gap $u_{(1,l,r)}$, for $l=1,2,\dots,M$, \emph{and} faces a gap that is greater than $u_{(1,l,r)}-y$. In this case, the customer's service time is $\Delta_{r}$. For the case $j\geq2$ we find
\begin{align}
\tilde{G}_{(j,l,r)}(s,y)
&=p_{r}\sum_{l_1=1}^{M}p_{(1,l_1,r)}1_{\{u_{(1,l_1,r)}> y\}}\int_0^{u_{(1,l_1,r)}-y}qe^{-qt} e^{-s(y+t)} \mathbb{E}\left[e^{-sG^{(n)}}1_{\{J_{n+1}=(j,l,r)\}}|E^{(n)}_{r}\right]{\rm d}t \nonumber\\
 &=p_{r}\sum_{l_1=1}^{M}p_{(1,l_1,r)}1_{\{u_{(1,l_1,r)}> y\}} \frac{qe^{-sy}}{s+q}(1-e^{-(s+q)(u_{(1,l_1,r)}-y)})\mathbb{E}\left[e^{-sG^{(n)}}1_{\{J_{n+1}=(j,l,r)\}}|E^{(n)}_{r}\right]\nonumber\\
&=p_{r}\sum_{l_1=1}^{M}p_{(1,l_1,r)} \frac{qe^{-sy}}{s+q}(1-e^{-(s+q)\text{max}\{u_{(1,l_1,r)}-y,0\}})\mathbb{E}\left[e^{-sG^{(n)}}1_{\{J_{n+1}=(j,l,r)\}}|E^{(n)}_{r}\right]\nonumber\\
&=p_{r}\frac{qe^{-sy}}{s+q}\left(1-\E\left[e^{-(s+q)\text{max}\{T_{(1,r)}-y,0\}}\right]\right)\mathbb{E}\left[e^{-sG^{(n)}} 1_{\{J_{n+1}=(j,l,r)\}}|E^{(n)}_{r}\right];\nonumber
\end{align}
here it should be kept in mind that
\[\mathbb{E}\left[e^{-sG^{(n)}} 1_{\{J_{n+1}=(j,l,r)\}}|E^{(n)}_{r}\right]\]
does not depend on $n$ and can be rewritten as
\begin{align}
\lefteqn{ \left( \prod\limits_{m=2}^{j-1}\sum_{k=1}^{M}p_{(m,k,r)} \int_0^{u_{(m,k,r)}}qe^{-qt}e^{-st}{\rm d}t \right) p_{(j,l,r)} \int_{u_{(j,l,r)}}^\infty qe^{-qt}e^{-s\Delta_{r}}{\rm d}t}   \nonumber\\
&=p_{(j,l,r)}e^{-(s\Delta_{r}+qu_{(j,l,r)})} \left(\frac{q}{s+q}\right)^{j-2}\prod\limits_{m=2}^{j-1}\sum_{k=1}^{M}p_{(m,k,r)}(1-e^{-(s+q)u_{(m,k,r)}})\nonumber\\
&=p_{(j,l,r)}e^{-(s\Delta_{r}+qu_{(j,l,r)})} \left(\frac{q}{s+q}\right)^{j-2}\prod\limits_{m=2}^{j-1}\left(1-\sum_{k=1}^{M}p_{(m,k,r)}e^{-(s+q)u_{(m,k,r)}}\right)\nonumber\\
&=p_{(j,l,r)}e^{-(s\Delta_{r}+qu_{(j,l,r)})} \left(\frac{q}{s+q}\right)^{j-2}\prod\limits_{m=2}^{j-1}\left(1-\E\left[e^{-(s+q)T_{(m,r)}}\right]\right).\nonumber
\end{align}
Slightly rewriting this last expression completes the proof of Lemma \ref{lemmaGsy}.
\end{proof}

The next lemma shows how to compute the probabilities $\pi_{(i,k,r)}$. We denote by $\tau_q$ the generic interarrival time between two vehicles on the \emph{major} road, which is exponentially distributed with parameter $q$. Let $\hat v$ be defined as $\max\{v,0\}$, and $\check v$   as $\max\{-v,0\}$.  In addition, $v^{(1,k,r)}_{(i,l,r_0)}=u_{(1,k,r)}-u_{(i,l,r_0)}+\Delta_{r_0}$.

\begin{lemma}\label{lemmapis}
The probability that a driver of profile $r$ is served in his $i$-th attempt with critical headway  $u_{(i,k,r)}$
is given by
\begin{align}
\pi_{(i,k,r)}=P_{(i,k,r)}\prod\limits_{m=1}^{i-1}\left(1-\sum_{k_m=1}^Mp_{(m,k_m,r)}P_{(m,k_m,r)}\right),\label{pi_i_BB}
\end{align}
where, for $i\in\{2,3,\dots,N\}$ and $k\in\{1,2,\dots,M\}$,
\begin{align}
P_{(i,k,r)}=\mathbb{P}(\tau_q\geq u_{(i,k,r)})=e^{-qu_{(i,k,r)}}. \label{Pj_BB}
\end{align}
The remaining  $P_{(1,k,r)}$ can be found by solving the system of linear equations
\begin{align}\label{P(1,k)_eq_B}
&\Bigg(1 +p_{(1,k,r)}\left(-p_{r}e^{-q\hat v^{(1,k,r)}_{(1,k,r)}}+c_{(k,r,r)} \right) \Bigg)P_{(1,k,r)}\\
&+\sum_{\substack{r_0=1,\\r_0\neq r}}^R\Bigg(-p_{r_0}e^{-q\hat v^{(1,k,r)}_{(1,k,r_0)}}+c_{(k,r_0,r)} \Bigg)p_{(1,k,r_0)}P_{(1,k,r_0)}\nonumber\\
&+\sum_{r_0=1}^{R}\sum_{\substack{l_1=1,\\l_1\neq k}}^M\Bigg(-p_{r_0}e^{-q \hat v^{(1,k,r)}_{(1,l_1,r_0)}}+c_{(k,r_0,r)} \Bigg)p_{(1,l_1,r_0)}P_{(1,l_1,r_0)}  \nonumber\\
=&\sum_{r_0=1}^R\sum_{l=1}^M\sum_{i=1}^{N}\Bigg(1-e^{-\lambda\check v^{(1,k,r)}_{(i,l,r_0)}} -e^{-q\hat v^{(1,k,r)}_{(i,l,r_0)}}+\frac{\lambda}{\lambda+q}e^{-(\lambda+q)\check v^{(1,k,r)}_{(i,l,r_0)}-qv^{(1,k,r)}_{(i,l,r_0)}} \nonumber\\
&+\frac{q}{\lambda+q}e^{-\lambda\bar u_{(i,l,r_0)}-qu_{(1,k,r)}} \Bigg)f_{(i,l,r_0)}(0) +\sum_{r_0=1}^Rc_{(k,r_0,r)},\nonumber
\end{align}
for $k\in\{1,2,\dots,M\}$ and $ r\in\{1,2,\dots,R\}$,
where \[
c_{(k,r_0,r)}=\sum_{l=1}^M\sum_{i=2}^{N}p_{r_0}p_{(i,l,r_0)}e^{-q(u_{(i,l,r_0)}+\hat v^{(1,k,r)}_{(i,l,r_0)})} \prod\limits_{m=2}^{i-1}\left(1-\E[e^{-qT_{(m,r_0)}}]\right).
\]
\end{lemma}

\begin{proof}
Expression \eqref{pi_i_BB} simply follows from the definition of $\pi_{(i,k,r)}$, by multiplying the probabilities that the driver does \emph{not} succeed in attempts $1, 2, \dots, i-1$ (each time distinguishing between all $M$ possible random critical gap values) and finally multiplying with the probability that he succeeds in the $i$-th attempt with critical headway $u_{(i,k,r)}$.

For $i\in\{2,3,\ldots,N\}$ it is easy to determine $P_{(i,k,r)}$ because we do not have to take into account any gap left by the predecessor; we simply compute the probability that the critical gap $u_{(i,k,r)}$ is smaller than the remaining interarrival time $\tau_q$ (which, due to the memoryless property, is again exponentially distributed with rate $q$).

The case $i=1$ is considerably more complicated.
Using \eqref{P(J=(i,k))_B_2} and noting that $\pi_{(1,k,r)}=P_{(1,k,r)}$, we can find a system of equations for $P_{(1,k,r)}$ by conditioning on the type of the \emph{predecessor} of the vehicle under consideration.
Define
\[\chi_{(i,l,r_0)}^{(j,k,r)}:= \P(J_n=(j,k,r)|J_{n-1}=(i,l,r_0)),\]
and
\begin{align}\phi_{(i,l,r_0)}^{(j,k,r)}&\,:= \P(J_n=(j,k,r)|J_{n-1}=(i,l,r_0),X_{n-2}=0),\\
\psi_{(i,l,r_0)}^{(j,k,r)}&\,:= \P(J_n=(j,k,r)|J_{n-1}=(i,l,r_0),X_{n-2}\ge 1).\end{align}
Evidently,
\begin{align}
\label{first_empty_B}
P_{(1,k,r)}=&\frac{1}{p_{(1,k,r)}p_r}\P(J_n=(1,k,r))\\
=&\frac{1}{p_{(1,k,r)}p_r}\sum_{r_0=1}^R\sum_{l=1}^M\sum_{i=1}^{N}
\chi_{(i,l,r_0)}^{(1,k,r)}\,\P(J_{n-1}=(i,l,r_0)) \nonumber\\
=&\frac{1}{p_{(1,k,r)}p_r}\sum_{r_0=1}^R\sum_{l=1}^M\sum_{i=1}^{N}\chi_{(i,l,r_0)}^{(1,k,r)}\,\pi_{(i,l,r_0)}p_{(i,l,r_0)}p_{r_0} \nonumber\\
=&\frac{1}{p_{(1,k,r)}p_r}\sum_{r_0=1}^R\sum_{l=1}^M\sum_{i=1}^{N}\Big(\phi_{(i,l,r_0)}^{(1,k,r)}\,\P(X_{n-2}=0|J_{n-1}=(i,l,r_0))\,+\nonumber\\
&\:\:\:\:\:\psi_{(i,l,r_0)}^{(1,k,r)}\,\P(X_{n-2}\geq 1|J_{n-1}=(i,l,r_0))\Big)\pi_{(i,l,r_0)}p_{(i,l,r_0)}p_{r_0}  \nonumber,
\end{align}
which can be further evaluated, using notation introduced before, as
\begin{align}
&\frac{1}{p_{(1,k,r)}p_r}\sum_{r_0=1}^R\sum_{l=1}^M\sum_{i=1}^{N}
\Bigg(\phi_{(i,l,r_0)}^{(1,k,r)}\, {\bar f_{(i,l,r_0)} } +\psi_{(i,l,r_0)}^{(1,k,r)}\left(1-\bar f_{(i,l,r_0)}\right)\Bigg)\bar \pi_{(i,l,r_0)} , \nonumber
\end{align}
with $\bar \pi_{(i,l,r_0)}:=
\pi_{(i,l,r_0)}p_{(i,l,r_0)}p_{r_0}$.
We now consider the individual terms separately. Observe that, by conditioning on the value of $\tau_q$,
\begin{align}
\frac{1}{p_{(1,k,r)}p_r}\phi_{(i,l,r_0)}^{(1,k,r)}\, {\bar f_{(i,l,r_0)} }\bar \pi_{(i,l,r_0)}&=\Bigg(\int_{0}^{\max\{\bar u_{(i,l,r_0)}-u_{(1,k,r)},0\}}\lambda e^{-\lambda x}{\rm d}x\nonumber\\ &+\int_{\max\{\bar u_{(i,l,r_0)}-u_{(1,k,r)},0\}}^{\bar u_{(i,l,r_0)}}\mathbb{P}(\tau_q \geq u_{(1,k,r)}-\bar u_{(i,l,r_0)}+x)\lambda e^{-\lambda x}{\rm d}x \nonumber\\
&+ \int_{\bar u_{(i,l,r_0)}}^{\infty}\mathbb{P}(\tau_q \geq u_{(1,k,r)})\lambda e^{-\lambda x}{\rm d}x \Bigg)f_{(i,l,r_0)}(0)\nonumber.
\end{align}
Also,
\begin{align*}
\frac{1}{p_{(1,k,r)}p_r}&\psi_{(i,l,r_0)}^{(1,k,r)}\, (1-{\bar f_{(i,l,r_0)} })\bar \pi_{(i,l,r_0)}\\
&=\,\mathbb{P}\left(\tau_q \geq \max\{u_{(1,k,r)}-\bar u_{(i,l,r_0)},0\}\right)\left(\pi_{(i,l,r_0)}p_{(i,l,r_0)}p_{r_0}-f_{(i,l,r_0)}(0)\right)
.\end{align*}
Combining the above, it follows that $P_{(1,k,r)}$ equals, with $v^{(1,k,r)}_{(i,l,r_0)}$ as defined above,
\begin{align}
\sum_{r_0=1}^R&\sum_{l=1}^M\sum_{i=1}^{N}\Bigg(1-e^{-\lambda \check v^{(1,k,r)}_{(i,l,r_0)}} -e^{-q\hat v^{(1,k,r)}_{(i,l,r_0)}}+\frac{\lambda}{\lambda+q}e^{-(\lambda+q)\check v^{(1,k,r)}_{(i,l,r_0)}-qv^{(1,k,r)}_{(i,l,r_0)}} \nonumber\\
&+\frac{q}{\lambda+q}e^{-\lambda(u_{(i,l,r_0)}-\Delta_{r_0})-qu_{(1,k,r)}} \Bigg)f_{(i,l,r_0)}(0)+\sum_{r_0=1}^R\sum_{l=1}^M\sum_{i=1}^{N}e^{-q\hat v^{(1,k,r)}_{(i,l,r_0)}}\pi_{(i,l,r_0)} p_{(i,l,r_0)}p_{r_0}.\nonumber
\end{align} \\
These expressions can be written in a more convenient form. To this end,
it is noted that
\begin{align*}
\lefteqn{\sum_{r_0=1}^R\sum_{l=1}^M\sum_{i=1}^{N}e^{-q \hat v^{(1,k,r)}_{(i,l,r_0)}}\pi_{(i,l,r_0)}p_{(i,l,r_0)}p_{r_0}}\\&=\sum_{r_0=1}^R\sum_{l=1}^Me^{-q\hat v^{(1,k,r)}_{(1,l,r_0)}} P_{(1,l,r_0)}p_{(1,l,r_0)}p_{r_0}
+\sum_{r_0=1}^R\sum_{l=1}^M\sum_{i=2}^{N}e^{-q\hat v^{(1,k,r)}_{(i,l,r_0)}}\pi_{(i,l,r_0)}p_{(i,l,r_0)}p_{r_0} \nonumber\\
&=\sum_{r_0=1}^R\sum_{l=1}^Me^{-q\hat v^{(1,k,r)}_{(1,l,r_0)}}P_{(1,l,r_0)}p_{(1,l,r_0)}p_{r_0}\nonumber\\
&+\sum_{r_0=1}^R\sum_{l=1}^M\sum_{i=2}^{N}e^{-q\hat v^{(1,k,r)}_{(i,l,r_0)}}p_{(i,l,r_0)}p_{r_0}P_{(i,l,r_0)} \prod\limits_{m=1}^{i-1}\left(1-\sum_{l_m=1}^Mp_{(m,l_m,r_0)}P_{(m,l_m,r_0)}\right).
\end{align*}
After some basic algebraic manipulations, and using \eqref{Pj_BB} we obtain
\begin{align}\label{one_B}
&\sum_{r_0=1}^R\sum_{l=1}^M\sum_{i=1}^{N}e^{-q \hat v^{(1,k,r)}_{(i,l,r_0)}}\pi_{(i,l,r_0)}p_{(i,l,r_0)}p_{r_0}=\sum_{r_0=1}^R\sum_{l_1=1}^M\Bigg( e^{-q\hat v^{(1,k,r)}_{(1,l_1,r_0)}}\\
&-\sum_{l=1}^M\sum_{i=2}^{N}p_{(i,l,r_0)}e^{-q(u_{(i,l,r_0)}+\hat v^{(1,k,r)}_{(i,l,r_0)})}\prod\limits_{m=2}^{i-1} \left(1-\sum_{l_m=1}^Mp_{(m,l_m,r_0)}e^{-qu_{(m,l_m,r_0)}}\right)\Bigg)P_{(1,l_1,r_0)}p_{(1,l_1,r_0)}p_{r_0} \nonumber\\
&+\sum_{r_0=1}^R\sum_{l=1}^M\sum_{i=2}^{N}p_{r_0}p_{(i,l,r_0)}e^{-q(u_{(i,l,r_0)}+\hat v^{(1,k,r)}_{(i,l,r_0)})}\prod\limits_{m=2}^{i-1}\left(1-\sum_{l_m=1}^Mp_{(m,l_m,r_0)}e^{-qu_{(m,l_m,r_0)}}\right). \nonumber
\end{align}
Using Equation \eqref{one_B} in \eqref{first_empty_B}, and after simplification, we obtain the linear system of equations \eqref{P(1,k)_eq_B}, which proves Lemma \ref{lemmapis}.
\end{proof}

\begin{remark}\label{remark: sufficient cond_supermodel}
We now discuss in more detail Assumption \ref{assumptiongapuse}, and how it plays a role in our result. The assumption entails  that the gap $T_{(i,r_0)}-\Delta_{r_0}$ left by a driver of profile $r_0$, while being successful in the $i$-th attempt, can be used by the subsequent driver only.
If we would allow more than one following vehicle use a gap left behind by a merging vehicle, the analysis becomes much more complicated due to the fact that we need to distinguish between all possible cases where multiple drivers fit into this lag.

For example, in the proof of Theorem \ref{thm1} we use explicitly that, if the system was not empty at arrival time, the gap left behind by a predecessor of type $(j, l, r_1)$ is simply $\bar u_{(j,l,r_1)} = u_{(j,l,r_1)}-\Delta_{r_1}$. However, a situation where this is not necessarily true, occurs when $j=1$  (meaning that the predecessor succeeded at his first attempt) and required a critical gap of $u_{(1,l,r_1)}$ that is \emph{smaller} than the remaining gap of the \emph{predecessor's predecessor}, say $\bar u_{(i,k,r_0)}$. In this case, the gap left behind by the predecessor is equal to
\[
u_{(i,k,r_0)} - \Delta_{r_0} - \Delta_{r_1} > u_{(1,l,r_1)} - \Delta_{r_1},
\]
meaning that the current vehicle and its predecessor both used the same gap $u_{(i,k,r_0)}$.
In this sense, Assumption \ref{assumptiongapuse} ensures the tractability of the model.

Assumption \ref{assumptiongapuse} has not been stated in terms of the model input parameters, and is therefore not straightforward to check. To remedy this, we give an equivalent definition in terms of the critical headway and the merging time. Assumption \ref{assumptiongapuse} is satisfied if and only if
\begin{align}\label{condition: B}
u_{(1,l,r_1)}\geq \bar u_{(i,k,r_0)}= u_{(i,k,r_0)}-\Delta_{r_0},
\end{align}
for all $i\in\{1,2,\dots,N\}$, $k,l\in\{1,2,\dots,M\}$, and $r_0,r_1\in\{1,2,\dots,R\}$.
In the next section, we will numerically show that in cases that Assumption \ref{assumptiongapuse}
is not fulfilled, the estimated capacity of the minor road is slightly smaller, but still highly accurate.
The reason why the estimated capacity is a \emph{lower} bound for the true capacity, is the fact that we waste some capacity by not allowing more than one vehicle to use the remaining part of a critical gap. The astute reader might have noticed that the maximum operators in Lemma~\ref{lemmaGsy} are not really needed due to condition \eqref{condition: B}. It turns out, however, that by preventing the corresponding terms from becoming negative, the capacity is approximated much better even when the condition is violated.
\end{remark}

\section{Numerical results}\label{numerics}

The analysis from the previous sections facilitates the evaluation of the performance of the system, including the assessment of the sensitivity of the capacity when varying model parameters. In this section, we present numerical examples to demonstrate the impact of these model parameters and of different driver behavior.

\subsection{Example 1:} In this illustrative example, we distinguish between two driver profiles. Profile~1 represents `standard' traffic, whereas profile 2 can be considered as `slower' traffic (for example large, heavily loaded vehicles). We assume that the ratio between profile 1 and profile 2 vehicles is $90\%/10\%$. The fact that profile 1 vehicles are faster than profile 2 vehicles is captured in their merging  times (respectively $\Delta_1$ and $\Delta_2$) and in the length of their critical gaps.
From the profile 1 drivers, we assume that 40\% need a gap of at least 5 seconds, upon arrival at the intersection, while the remaining 60\% need a gap of at least 6 seconds. If, however, the drivers do not find an acceptable gap right away, they will grow impatient and be more and more prepared to accept slightly smaller gaps. We introduce an impatience rate $\alpha\in(0,1)$ that determines how fast critical gaps decrease. Profile 2 vehicles are assumed to be slower, meaning that they have a need critical gaps of respectively 8 seconds (50\%) or 9 seconds (50\%) at their first attempt. Summarizing, we have the following model parameters in this example:

\[
\begin{array}{|r|r|}
\hline
\multicolumn{1}{|c|}{\text{Profile 1}} & \multicolumn{1}{|c|}{\text{Profile 2}} \\
\hline
p_1 = 0.9 & p_2=0.1\\
\hline
p_{i,1,1}=0.4 & p_{i,1,2}=0.5 \\
p_{i,2,1}=0.6 & p_{i,2,2}=0.5 \\
\hline
u_{1,1,1}= 5.0 & u_{1,1,2}= 8.0 \\
u_{1,2,1}= 6.0 & u_{1,2,2}= 9.0 \\
\hline
\end{array}\\[1ex]
\]
and, due to the impatience, we have for both profiles:
\begin{align}\label{sequence T_i}
u_{(i+1,k,r)}=\alpha(u_{(i,k,r)}-\Delta_r)+\Delta_r, \quad i=1,2,\dots,N-1,~ k=1,2,~ r=1,2.
\end{align}
We vary $\alpha$, $\Delta_1$ and $\Delta_2$ to gain insight into the impact of these parameters on the capacity of the minor road, while also varying $q$, the arrival rate on the major road. We take $\alpha\in\{0.6,0.8,0.9,1.0\}$, $\Delta_1\in\{4,5\}$ and $\Delta_2\in\{5,6,7\}$ and vary $q$ between 0 and 1000 vehicles per hour.
Note that $\alpha=1$ corresponds to a model \emph{without} impatience. The results are depicted in Figures \ref{fig:example1}(a) and \ref{fig:example1}(b).
In Figure \ref{fig:example1}(a), we can observe how the capacity of the minor road increases when merging time corresponding to at least one of the driver profiles decreases. From Figure \ref{fig:example1}(b) we conclude by how much the capacity of the minor road increases when drivers become more impatient. The biggest capacity gain is caused by a decrease in $\Delta_1$ from 5 to 4, obviously because profile 1 constitutes the vast majority of all vehicles.

\begin{figure}[!ht]
\parbox{0.49\linewidth}{\centering
\includegraphics[width=\linewidth]{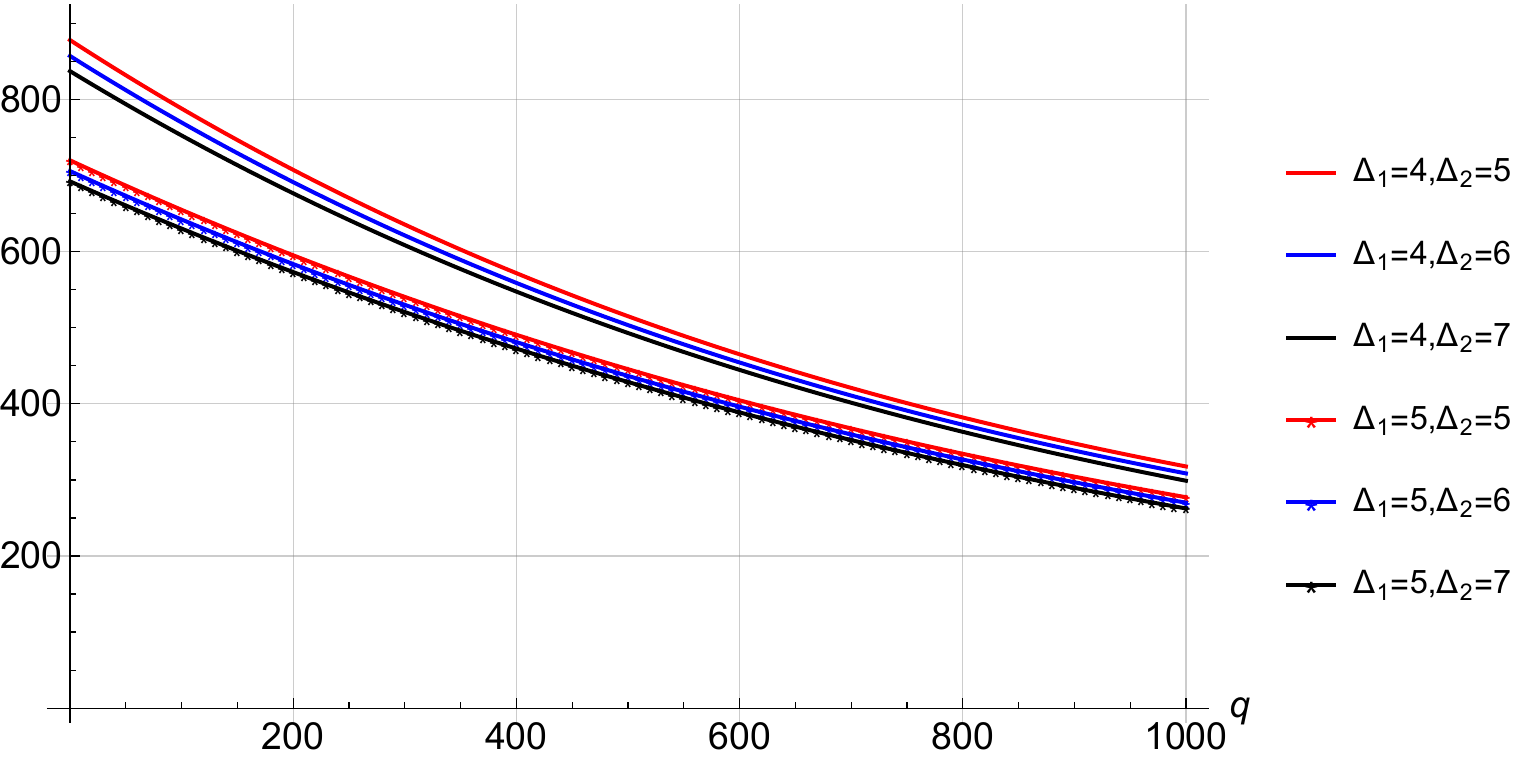}\\
\scriptsize (a) Impact of merging times of drivers for $\alpha=0.9$. }
\parbox{0.49\linewidth}{\centering
\includegraphics[width=\linewidth]{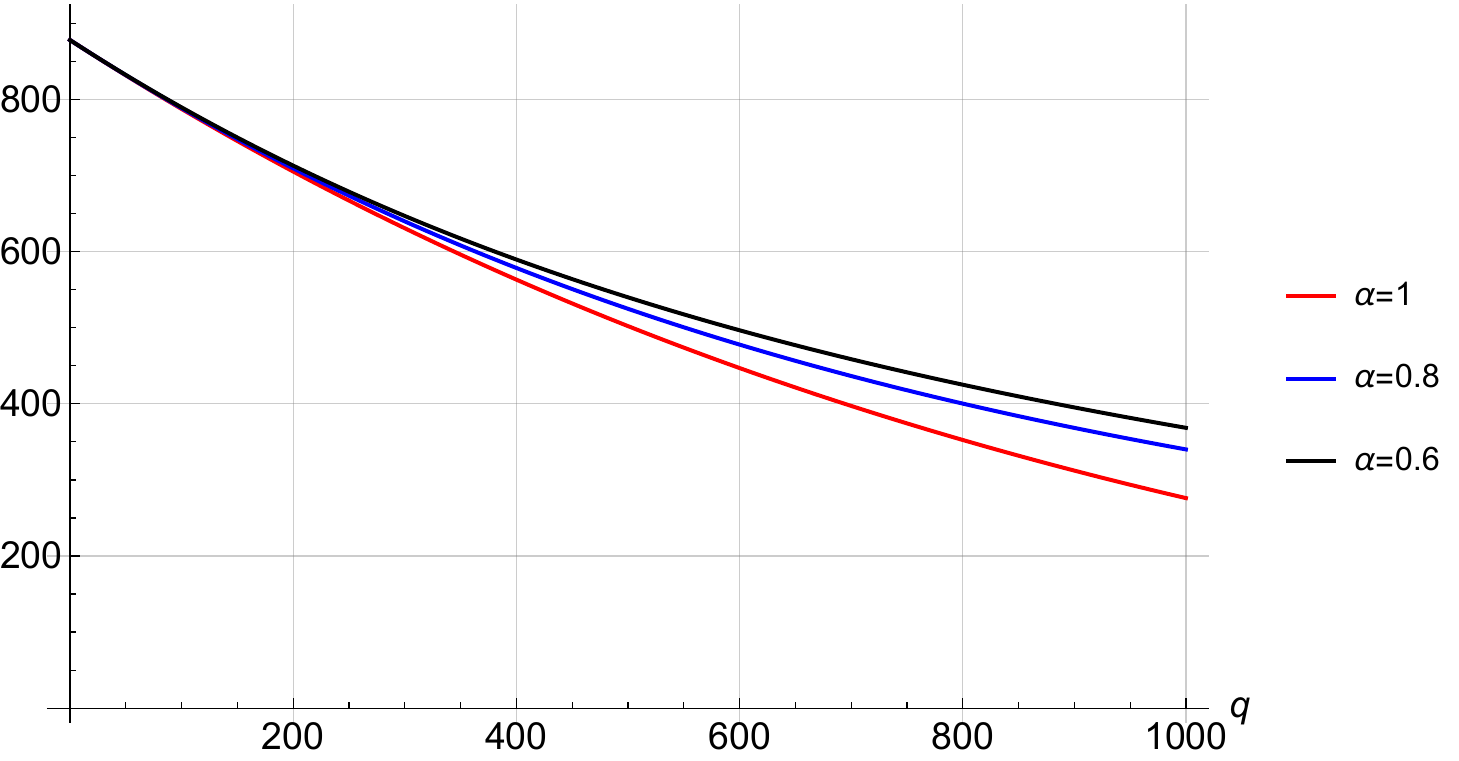}\\
\scriptsize (b) Impact of impatient behaviors of drivers for $\Delta_1=4$ seconds and $\Delta_2=5$ seconds}
\caption{Capacity of the minor road (veh/h) as a function of the flow rate on the major road (veh/h) in Example 1.}
\label{fig:example1}
\end{figure}

\subsection{Example 2:} The model parameters in Example 1 are such that condition \eqref{condition: B} is met. In practice, however, situations might arise where this is not the case. The purpose of this example is to quantify errors made when computing the capacity using the methods proposed in Section \ref{capacity on minor road}, since these methods rely on condition \eqref{condition: B}.  We take the same settings as in  Example 1, but we adjust the critical gaps  for profile 2 drivers from 8 and 9 to, respectively, 10 and 12 seconds. We keep the original merging times for both profile drivers of $\Delta_1=4$ seconds and $\Delta_2=5$ seconds. In order to quantify the errors, we use both the analysis from Section \ref{capacity on minor road} and a computer program that simulates the model and gives accurate results.
Table \ref{table:1} compares the capacities obtained with both methods, for various values of $q$ and $\alpha$. It immediately becomes apparent that the relative error is below $0.5$ percent in all cases. We conducted similar experiments for different settings and in all cases we obtained relative errors below one percent, which clearly justifies using the analysis of this paper in all practical situations, including those violating \eqref{condition: B}. As expected (see Remark \ref{remark: sufficient cond_supermodel}), the approximated capacities are (slightly) underestimating the actual capacities.


\begin{table}[h!]
\[
\begin{array}{|l|l|c|c|c|c|}
  \hline
   &   &  q=250 & q=500 & q=750 & q=1000 \\
  \hline
\alpha=1.0   & \text{Approximation} & 646.2 & 466.4  & 328.9  & 225.8  \\
    & \text{Simulation} & 647.2& 467.7 & 330.0 & 226.5 \\
    \hline
\alpha=0.9   & \text{Approximation} & 652.8  & 491.0  & 377.8 & 298.9  \\
   & \text{Simulation} & 653.7  & 491.5 & 378.0 & 299.0 \\
  \hline
\end{array}
\]\vspace{2mm}
\caption{Capacity of the minor road (veh/h) for various values of the flow rate, $q$, on the major road (veh/h) in Example 2.}
\label{table:1}
\end{table}

\section{Discussion and concluding remarks}

In this paper we have presented a gap acceptance model for unsignalized intersections that considerably generalizes the existing models.  The model proposed incorporates various realistic aspects that were not taken care of in previous studies: driver impatience, heterogeneous driving behavior, and the service time being dependent on the vehicle arriving at an empty queue or not. Despite the rather intricate system dynamics, we succeed in providing explicit expressions for the stationary queue-length distribution of vehicles on the minor road, which facilitate the evaluation of the corresponding capacity (of the minor road, that is).

We have concluded our paper by presenting a series of numerical results, which are representative of the extensive experiments that we performed. Our techniques facilitate the quantitative evaluation of the system at hand, including the assessment of the sensitivity of the capacity when varying the model parameters (which is typically considerably harder when relying on simulation).

An interesting direction in which to pursue this line of research, would be to allow a more general arrival process on the major road. In the present paper, we assume that the arrival process on the major road is a Poisson process. However, in practice there might be platoon forming on this road and several papers have shown that this clustering of vehicles will influence the capacity of the \emph{minor} road (cf. \cite{AbhishekMMPPandImpatience,tanner62,wegmann1991,wu2001}). It would certainly be interesting to combine the framework in this paper with, for example, the gap-block arrival process in Cowan \cite{cowan75} or Markov platooning introduced in \cite{AbhishekMMPPandImpatience}.

\subsection*{Acknowledgments} The research of Abhishek and M. Mandjes is partly funded by NWO Gravitation project {\sc Networks}, grant number 024.002.003. The authors thank Onno Boxma (Eindhoven University of Technology) and Rudesindo N\'u\~nez Queija  (University of Amsterdam) for their valuable input on earlier drafts of this paper.

\bibliographystyle{abbrv}

\end{document}